\documentclass[reqno,12pt,epsf]{amsart}
\usepackage{amsmath}
\usepackage{latexsym,amsmath}
\usepackage{amsthm}
\usepackage{amssymb}
\usepackage{graphics}

\newtheorem{theorem}{Theorem}     

\newtheorem{conjecture}[theorem]{Conjecture}

\def\om{\omm}
\def\omm{\omega}

\def\R{\mathbb{R}}

\def\per{\mbox{per}}
\def\deg{\mbox{deg}}

\newcounter{rot}

\title[Total weight choosability in Hypergraphs]
{Total weight choosability in Hypergraphs}

\author[F. Pfender]{Florian Pfender}
\curraddr[F. Pfender]{University of Colorado Denver\\ Department of Mathematics and Statistics\\ Denver, CO, USA}
\email{Florian.Pfender@ucdenver.edu}
\thanks{Research supported in part by Simon's Foundation Collaboration Grant}

\begin{document}

\begin{abstract}
A total weighting of the vertices and edges of a hypergraph is called vertex-coloring if the total weights of the vertices
yield a proper coloring of the graph, i.e., every edge contains at least two vertices with different weighted degrees. In this note we show that such a weighting is possible
if every vertex has two, and every edge has three weights to choose from, extending a recent result on graphs to hypergraphs.
\end{abstract}

\maketitle

\section{Introduction and Notation}
In 2004, Karo\'{n}ski, {\L}uczak and Thomason~\cite{KLT} made an innocent looking conjecture on edge weightings in graphs, which attracted a lot of activity in the following years.
\begin{conjecture}
For every graph $G$ without isolated edges, there is a weighting $\omm: E(G)\to \{1,2,3\}$,
such that the induced vertex weights $\omm(v):=\sum_{u\in N(v)}\omm(uv)$ properly color $V(G)$.
\end{conjecture}
This so called 1-2-3-Conjecture is known to be true for several classes of graphs, the best known result for general graphs is given in~\cite{KKP}.
\begin{theorem}\label{125}
For every graph $G$ without isolated edges, there is a weighting $\omm: E(G)\to \{1,2,3,4,5\}$,
such that the induced vertex weights $\omm(v):=\sum_{u\in N(v)}\omm(uv)$ properly color $V(G)$.
\end{theorem}
Shortly thereafter, a total version of the 1-2-3-Conjecture, adaptly called the 1-2-Conjecture, was formulated by Przyby{\l}o and Wozniak~\cite{PW}.
\begin{conjecture}\label{12con}
For every graph $G$, there is a weighting $\omm: E(G)\cup V(G)\to \{1,2\}$,
such that the induced total vertex weights $w(v):=\omm(v)+\sum_{u\in N(v)}\omm(uv)$ properly color $V(G)$.
\end{conjecture}
Kalkowski in~\cite{K} came close to settling this conjecture.
\begin{theorem}\label{tot123}
For every graph $G$, there are weightings $\omm: E(G)\to \{1,2,3\}$ and $\omm': V(G)\to \{1,2\}$
such that the induced total vertex weights $w(v):=\omm'(v)+\sum_{u\in N(v)}\omm(uv)$ properly color $V(G)$.
\end{theorem}
One natural and promising approach for both conjectures is the use of Alon's Combinatorial Nullstellensatz (see~\cite{A}). In its most straightforward application, it would prove list versions of the conjectures if successful, leading to the following stronger conjectures, first stated by Bartnicky, Grytczuk and Niwczyk, and by Przyby{\l}o and Wozniak, and Wong and Zhu, respectively.
\begin{conjecture}\label{123list}\cite{BGN}
For every graph $G$ without isolated edges, and for every assignment of lists of size $3$ to the edges of $G$, there exists a weighting $\omm: E(G)\to \R$ from the lists,
such that the induced vertex weights $\omm(v):=\sum_{u\in N(v)}\omm(uv)$ properly color $V(G)$.
\end{conjecture}
\begin{conjecture}\label{12list}\cite{PW2},\cite{WZ}
For every graph $G$, and for every assignment of lists of size $2$ to the vertices and edges of $G$, there exists a weighting $\omm: V(G)\cup E(G)\to \R$ from the lists,
such that the induced total vertex weights $w(v):=\omm(v)+\sum_{u\in N(v)}\omm(uv)$ properly color $V(G)$.
\end{conjecture}
Conjecture~\ref{123list} is open even if we allow larger lists of some fixed size $k$. For Conjecture~\ref{12list}, the best result due to Zhu and Wong in~\cite{ZW} generalizes Theorem~\ref{tot123}.
\begin{theorem}\label{tot123list}
For every graph $G$, and for every assignment of lists of size $2$ to the vertices and of size $3$ to the edges of $G$, there exists a weighting $\omm: V(G)\cup E(G)\to \R$ from the lists,
such that the induced total vertex weights $w(v):=\omm(v)+\sum_{u\in N(v)}\omm(uv)$ properly color $V(G)$.
\end{theorem}
All these questions also make sense for hypergraphs. In this note, a vertex coloring of a hypergraph is proper if there is no monochromatic edge. A hypergraph is {\em linear} if any two edges intersect in at most one vertex. Two vertices in a hypergraph are {\em twins} if they lie in the excat same set of edges. With these notions, Theorem~\ref{125} was generalized in~\cite{KKP2}.
\begin{theorem}
For every linear hypergraph $H$ with edges of order at most $r\ge 2$, and no edge consisting of twins, there is a weighting $\omm: E(H)\to \{1,2,\ldots,\max\{5,r+1\}\}$,
such that the induced vertex weights $\omm(v):=\sum_{e\ni v}\omm(e)$ properly color $V(H)$.
\end{theorem}
This theorem is sharp for all $r\ge 4$. Further, for general hypergraphs the following bound is shown.
\begin{theorem}\label{Thm:gen}
For every hypergraph $H$ with all edges of order between $2$ and $r$, and no edge consisting of twins, there is a weighting $\omm: E(H)\to \{1,2,\ldots,5r-5\}$,
such that the induced vertex weights $\omm(v):=\sum_{e\ni v}\omm(e)$ properly color $V(H)$.
\end{theorem}
Using similar but simpler ideas than in the proofs of the previous two theorems together with the proof of Theorem~\ref{tot123}, one can easily show the hypergraph version of Theorem~\ref{tot123}.
\begin{theorem}
For every hypergraph $H$ without edges of size $0$ or $1$, there are weightings $\omm: E(G)\to \{1,2,3\}$ and $\omm': V(G)\to \{1,2\}$
such that the induced total vertex weights $w(v):=\omm'(v)+\sum_{e\ni v}\omm(e)$ properly color $V(H)$.
\end{theorem}
We will not present the proof here, as the statement is implied by the list version, our main result of this note, which generalizes Theorem~\ref{tot123list}.
\begin{theorem}\label{main}
For every hypergraph $H$ without edges of size $0$ and $1$, and for every assignment of lists of size $2$ to the vertices and of size $3$ to the edges of $G$, there exists a weighting $\omm: V(G)\cup E(G)\to \R$ from the lists,
such that the induced total vertex weights $w(v):=\omm(v)+\sum_{e\ni v}\omm(e)$ properly color $V(G)$.
\end{theorem}
\begin{proof}
The proof is an adaptation of the proof of Theorem~\ref{tot123list} to hypergraphs.
Order the vertices in some order.
For every edge $e\in E(H)$, let $u_e,v_e\in e$ be the first two vertices in the edge. We will show the stronger statement, that we can find a weighting such that $w(u_e)\ne w(v_e)$ for every edge $e\in E(H)$.
We may assume that $\{u_e,v_e\}\ne \{u_{e'},v_{e'}\}$ for any pair of edges $e,e'\in E(H)$. Otherwise, pick a weight from the list of $e'$, and add it to the lists of all vertices in $e'$. If we can now find a coloring weighting in the Hypergraph $H-e'$ from the new lists, the corresponding weighting in $H$ will also induce a proper coloring.

Now consider the polynomial
\[
\phi(\om)=\prod_{e\in E(H)}(w(u_e)-w(v_e)).
\]
Note that non-zeros of $\phi$ correspond exactly to total weightings of $H$ with $w(u_e)\ne w(v_e)$ for every edge $e\in E(H)$. Let $n=|V(H)|$, $m=|E(H)|$, and let $A_H$ be an $m\times (m+n)$ matrix, where the rows are labeled with the elements of $E(H)$, the columns are labeled with the elements of $E(H)\cup V(H)$, and
\[
A_H(e,z)=
\begin{cases}
~1,&\mbox{ if }z\in V(H)\mbox{ and }z\in e,\\
~1, &\mbox{ if }z\in E(H)\mbox{ and }u_e\in z,v_e\notin z,\\
-1, &\mbox{ if }z\in E(H)\mbox{ and }v_e\in z,u_e\notin z,\\
~0, &\mbox{ otherwise}.
\end{cases}
\]
Let $B$ be an $m\times m$ matrix consisting of columns of $A_H$, such that every column corresponding to a vertex of $H$ appears at most once, and every column corresponding to an edge of $H$ appears at most twice. Then the permanent $\per(B)$ equals a coefficient of a maximum degree monomial in $\phi$. By a standard application of Alon's Combinatorial Nullstellensatz, there exists a non-zero of $\phi$ (and thus a proper coloring by total vertex weights from the lists) from any given list assignment with vertex lists of size two and edge lists of size three, if we can find such a matrix $B$ with $\per(B)\ne 0$.

We will find such a $B$ by induction on $n$. For $n=0$, the statement is trivial as by definition, the permanent of a $0\times 0$ matrix is $1$. For $n\ge 1$, let $u\in V(H)$ be the first vertex in the order, and consider the hypergraph $H'$ induced by $H$ on $V(H)\setminus\{u\}$, i.e., $H'$ contains all edges of $H$ which do not contain $u$. Let $k=\deg(u)$, then by induction there is an $(m-k)\times (m-k)$-matrix $B'$ with $\per(B')\ne 0$ consisting of columns of $A_{H'}$, at most two columns equal to $A_{H'}(.,e)$ for each edge, and at most one column equal to $A_{H'}(.,v)$ for each vertex.

Now build an $m\times m$ matrix $C$ from $B'$ by the use of the corresponding columns of $A_H$, and the addition of $k$ identical columns equal to $A_H(.,u)$. These added columns all have $k$ entries equal to $1$, and the remaining entries $0$, so $\per(C)=k!~\per(B')\ne 0$. Now, whenever $u\in e\in E(H)$, and $C$ contains $A_H(.,v_e)$, replace that column by the column $A_H(.,e)$. Note that $A_H(.,e)=A_H(.,u)-A_H(.,v_e)$. By the multilinearity of the permanent, the difference of permanents of $C$ and the matrix after the switch is equal to the permanent of a matrix containing $(k+1)$ copies of $A_H(.,u)$, a singular matrix with permanent $0$. Thus, we can make all these switches one-by-one, arriving at a matrix $D$ without columns equal to $A_H(.,v_e)$ and at most one column equal to $A_H(.,e)$ for $u\in e\in E(H)$, and with $\per(D)=\per(C)$.

To now get $B$ from $D$, we will replace the columns equal to $A_H(.,u)$ appropriately one-by-one. For every $e\ni u$, replacing one column $A_H(.,u)$ by one of $A_H(.,v_e)$ and $A_H(.,e)$ will result in a matrix with $\per\ne 0$. If both the resulting matrices had permanent $0$, then by the multilinearity of the permanent, the matrix before the replacement would have had permanent $0$, a contradiction. After replacing all columns equal to $A_H(.,u)$ this way, we arrive at $B$ with the desired properties.
\end{proof}

 \bibliographystyle{amsplain}

\end{document}